\documentclass{article}
\usepackage[T1]{fontenc}
\usepackage{titlesec}
\usepackage[utf8]{inputenc}
\usepackage{amsmath}
\usepackage{amssymb}
\usepackage{amsfonts}
\usepackage{amsthm}
\usepackage{geometry}
\usepackage{babel}
\usepackage{xcolor}
\newtheorem{lemma}{Lemma}
\newtheorem{theorem}{Theorem}

\newtheorem*{corollary}{Corollary}
\newtheorem*{remark}{Remark}
\numberwithin{equation}{section}
\usepackage{url}
\usepackage[hidelinks]{hyperref}

\title{PhD}
\author{}
\date{January 2021}

\begin{document}


\begin{center}
    \textbf{Moment Inequalities for Suprema of Gaussian Random Processes}
\end{center}

\begin{center}
     Simona Diaconu\footnote{Courant Institute, New York University, simona.diaconu@nyu.edu}
\end{center}

\begin{abstract}
    Suppose \((X_t)_{t \in T}\) is a Gaussian process indexed by some arbitrary set \(T:\) the random variable \(\sup_{t \in T}{X_t}\) can be very intricate and bounding its expectation is a natural step towards understanding it. Sudakov-Fernique inequality allows to order expectations of suprema of such random processes: if \((X_t)_{t \in T},(Y_t)_{t \in T}\) are centered Gaussian random processes satisfying \(\mathbb{E}[(X_t-X_s)^2] \leq \mathbb{E}[(Y_t-Y_s)^2]\) for all \(t,s \in T,\) then \(\mathbb{E}[\sup_{t \in T}{X_t}] \leq \mathbb{E}[\sup_{t \in T}{Y_t}].\) This work obtains similar results for higher moments under a slightly stronger condition than the one aforementioned.
\end{abstract}


\section{Introduction}

A special case of suprema of random processes (i.e., collections of random variables that are not necessarily independent) for which there exist elegant and powerful results is the Gaussian family: these are normal random variables \((X_t)_{t \in T},\) whose indices lie in some set \(T,\) and the mostly analyzed statistic for such collections of random variables is
\[\sup_{t \in T}{X_t}.\]
Slepian's inequality and Sudakov-Fernique's inequality are classical results that relate suprema of two alike processes: suppose \((X_t)_{t \in T},(Y_t)_{t \in T}\) are centered Gaussian processes. The aforementioned inequalities, stated as theorems \(7.2.1\) and \(7.2.11\) in Vershynin~\cite{vershynin}, are as follows.
\vspace{0.3cm}
\par
If \(\mathbb{E}[X_t^2]=\mathbb{E}[Y_t^2],\mathbb{E}[(X_t-X_s)^2] \leq \mathbb{E}[(Y_t-Y_s)^2]\) for all \(t,s \in T,\) then for all \(\tau \geq 0,\)
\begin{equation}\label{slepian}\tag{Sle}
    \mathbb{P}(\sup_{t \in T}{X_t} \geq \tau) \leq \mathbb{P}(\sup_{t \in T}{Y_t} \geq \tau).
\end{equation}
\par
If \(\mathbb{E}[(X_t-X_s)^2] \leq \mathbb{E}[(Y_t-Y_s)^2]\) for all \(t,s \in T,\) then
\begin{equation}\label{sudakovferniq}\tag{Sud-Fer}
    \mathbb{E}[\sup_{t \in T}{X_t}] \leq \mathbb{E}[\sup_{t \in T}{Y_t}].
\end{equation}
Similarly, Gordon's inequality (exercise \(7.2.14\) in Vershynin~\cite{vershynin}) deals with suprema over an index set that is a cross product: 
\vspace{0.3cm}
\par
If \(\mathbb{E}[X_{ut}^2]=\mathbb{E}[Y_{ut}^2],\mathbb{E}[(X_{ut}-X_{us})^2] \leq \mathbb{E}[(Y_{ut}-Y_{us})^2], \mathbb{E}[(X_{ut}-X_{vs})^2] \leq \mathbb{E}[(Y_{ut}-Y_{vs})^2]\) for \(t,s \in T, \newline u \ne v, u,v \in U,\) then for all \(\tau \geq 0,\)
\begin{equation}\label{gordon}\tag{Gor}
    \mathbb{P}(\inf_{u \in U}{\sup_{t \in T}}{X_{ut}} \geq \tau) \leq \mathbb{P}(\inf_{u \in U}{\sup_{t \in T}}{Y_{ut}} \geq \tau).
\end{equation}
These inequalities are oftentimes used as upper bounds (i.e., \(X\) is determined by a question of interest, whereas \(Y\) is designed for \(X\)), and converse results providing lower bounds exist as well: for instance, Sudakov's minoration inequality states for a centered Gaussian process \((X_t)_{t \in T}\) and \(\epsilon>0,\)
\[\mathbb{E}[\sup_{t \in T}{X_t}] \geq c\epsilon \sqrt{\mathcal{N}(T,d,\epsilon)},\]
where 
\(\mathcal{N}(T,d,\epsilon)\) is the minimal size of an \(\epsilon\)-net of the metric space \((T,d)\) for \(d(t,s)=(\mathbb{E}[(X_t-X_s)^2])^{1/2},\) and \(c>0\) is an absolute constant. It must be mentioned that a first step in the proofs of all the aforesaid results is assuming \(|T|<\infty:\) this is connected to \(\sup_{t \in T}{X_t}\) not being a priori a random variable (i.e., it might fail to be a measurable function), which is dealt with by the following definition
\[\mathbb{E}[\sup_{t \in T}{X_t}]=\sup_{S \subset T, |S|<\infty}{\mathbb{E}[\max_{s \in S}{X_s}]}\]
(see also chapter \(10\) in Tropp~\cite{troppec}). This justifies considering solely the case \(|T|<\infty\) in what follows: furthermore, since an analogous definition can be adopted for functions of the suprema, the main result in this work, Theorem~\ref{extsudfer}, can be reformulated using moments of suprema. 
\par
The above inequalities 
rely on a key property of standard normal random variables,
\[\mathbb{E}[f'(X)]=\mathbb{E}[Xf(X)]\]
for \(X\overset{d}{=} N(0,1),\) and any differentiable function \(f:\mathbb{R} \to \mathbb{R}\) for which both expectations are finite: this identity can be employed to prove a version of Gaussian interpolation (lemma \(7.2.7\) in Vershynin~\cite{vershynin}).

\begin{lemma}[Gaussian Interpolation, Vershynin~\cite{vershynin}]\label{interpollemma}
Suppose \(X,Y \in \mathbb{R}^d\) are independent random vectors\footnote{By a slight abuse of notation, any zero vector is denoted by \(0.\)} with \(X\overset{d}{=}N(0,\Sigma^X),Y\overset{d}{=}N(0,\Sigma^Y),\) and let
\[Z(u)=\sqrt{u} \cdot X+\sqrt{1-u} \cdot Y, \hspace{0.5cm} u \in [0,1].\]
Then for any twice differentiable function \(f:\mathbb{R}^d \to \mathbb{R},\)
\begin{equation}\label{interpol}\tag{GI}
    \frac{d}{du}\mathbb{E}[f(Z(u))]=\frac{1}{2}\sum_{1 \leq i,j \leq d}{(\Sigma^X_{ij}-\Sigma^Y_{ij})\mathbb{E}[\frac{\partial^2f}{\partial{x_i} \partial{x_j}}(Z(u))]},
\end{equation}
assuming all expectations involved are finite.
\end{lemma}

The aforesaid results ((\ref{slepian}),(\ref{sudakovferniq}),(\ref{gordon})) are usually derived by defining functions \(f:\mathbb{R}^d \to \mathbb{R}\) that approximate the function of interest (such as \(\max_{1 \leq i \leq d}{x_i},\chi_{\max_{1 \leq i \leq d}{x_i} <\tau}\)), and whose second-order partial derivatives share signs with the corresponding entries of the covariance matrix of the vectors involved: the latter property alongside (\ref{interpol}) ensures \(u \to \mathbb{E}[f(Z(u))]\) is nondecreasing, whereby 
\[\mathbb{E}[f(Y)]=\mathbb{E}[f(Z(0))] \leq \mathbb{E}[f(Z(1))]=\mathbb{E}[f(X)],\] 
while the former propagates this inequality to the desired statistic via the dominated convergence theorem. (It must be mentioned that this method of guaranteeing the derivative does not change sign is attributed to Kahane~\cite{kahane} in the course notes by Tropp~\cite{troppec}: see chapter \(10\) therein.)
\par
A simple observation allows a wider use of (\ref{interpol}): the condition that the right-hand side term is always nonnegative can be relaxed via a limiting procedure. This paper builds on this and derives moment inequalities for moments of suprema of Gaussian random processes, similar in spirit to (\ref{sudakovferniq}).

\begin{theorem}\label{extsudfer}
Suppose \(X\overset{d}{=}N(0,\Sigma^X),Y\overset{d}{=}N(0,\Sigma^Y)\) for \(\Sigma^X,\Sigma^Y \in \mathbb{R}^{k \times k}\) that satisfy
\begin{equation}\label{strcondd}
    2|\Sigma^X_{ij}-\Sigma^Y_{ij}| \leq \Sigma^Y_{ii}-\Sigma^X_{ii}+\Sigma^Y_{jj}-\Sigma^X_{jj}, \hspace{1cm} (1 \leq i,j \leq k).
\end{equation}
Then for all \(m \geq 1,\)
\[\mathbb{E}[(\max_{1 \leq i \leq k}{|X_i|})^m] \leq \mathbb{E}[(\max_{1 \leq i \leq k}{|Y_i|})^m] .\]
\end{theorem}

Before proving this result, a few remarks on the proof of the Sudakov-Fernique inequality are in order. It relies on the following simple observation, 
\[\lim_{\beta \to \infty}{\frac{1}{\beta}\log{(\sum_{1 \leq i \leq k}{e^{\beta a_i}})}}=\max_{1 \leq i \leq k}{a_i} \hspace{0.5cm} (a_1,a_2,\hspace{0.05cm}...\hspace{0.05cm},a_k \in \mathbb{R}),\] 
as well as on a special property\footnote{\(\log\) denotes the natural logarithm, i.e., \(\log{e}=1.\)} of \(f(x_1,x_2,\hspace{0.05cm}...\hspace{0.05cm},x_k)=\log{(\sum_{1 \leq i \leq k}{e^{x_i}})},\) namely,
\begin{equation}\label{summ}
    \sum_{1 \leq i \leq k}{\frac{\partial{f}}{\partial{x_i}}}=1
\end{equation}
(see Vershynin~\cite{vershynin} for details): the latter property together with the covariance condition 
\[\mathbb{E}[(X_i-X_j)^2] \leq \mathbb{E}[(Y_i-Y_j)^2] \hspace{0.5cm} (1 \leq i,j \leq k)\]
guarantees the derivative in (\ref{interpol}) is a sum of terms that are at most \(0.\) The crucial identity (\ref{summ}) fails to hold for \(f(x_1,x_2,\hspace{0.05cm}...\hspace{0.05cm},x_k)=\log{(\sum_{1 \leq i \leq n}{e^{|x_i|^m}})}\) when \(m \ne 1\) (a function that is not even everywhere differentiable), and this obstructs generalizing (\ref{sudakovferniq}) to higher moments of the suprema. This work relies on another family of functions to bypass this difficulty and obtain comparison results between higher moments of maxima: the starting point is observing that
\[\lim_{n\to \infty}{(\sum_{1 \leq i \leq k}{|a_i|^{n}})^{m/n}}=(\max_{1 \leq i \leq k}{|a_i|})^m \hspace{0.5cm} (m>0)\]
\par
The rest of the paper is organized as follows. Section~\ref{sect1} introduces a corollary of Theorem~\ref{extsudfer}, subsection~\ref{sect01} consists of the proof of Theorem~\ref{extsudfer}, and subsection~\ref{sect02} collects two lemmas on correlations and expectations involving Gaussian random vectors.

\section{Monotonicity via Gaussian Intrapolation}\label{sect1}

Under the notation in Theorem~\ref{extsudfer}, the variance condition for Sudakov-Fernique's inequality can be rewritten as 
\begin{equation}\label{condd}
    2(\Sigma^Y_{ij}-\Sigma^X_{ij}) \leq \Sigma^Y_{ii}-\Sigma^X_{ii}+\Sigma^Y_{jj}-\Sigma^X_{jj}, \hspace{1cm} (1 \leq i,j \leq k).
\end{equation}
In applications of Theorem~\ref{extsudfer} this can be the key inequality to look for alongside a gurantee for \(\Sigma^Y_{ij} \geq \Sigma^X_{ij}.\) The latter can usually be accomplished at the cost of at most a factor depending solely on \(m,\) which is the content of the following corollary.

\begin{corollary}
    Suppose \(X,Y \in \mathbb{R}^k\) are centered Gaussian random vectors with 
    \[\mathbb{E}[(X_i-X_j)^2] \leq \mathbb{E}[(Y_i-Y_j)^2] \hspace{0.5cm} (1 \leq i,j \leq k).\] 
    Then for all \(m \geq 1,\)
    \[\mathbb{E}[(\max_{1 \leq i \leq k}{|X_i|})^m] \leq 2^{m-1} \cdot (\sigma^m\mathbb{E}[|g|^m]+\mathbb{E}[(\max_{1 \leq i \leq k}{|Y_i|})^m])\]
    for \(g\overset{d}{=}N(0,1),\) and \(\sigma=\max_{1 \leq i \leq k}{\sqrt{\mathbb{E}[X_i^2]+\mathbb{E}[Y_i^2]}}.\)
\end{corollary}

\begin{proof}
    Let 
    \[\Tilde{Y}=[(Y_i+\sigma g)_{1 \leq i \leq k}] \in \mathbb{R}^k\]
with \(g\overset{d}{=} N(0,1)\) independent of \(Y.\) Then \(X\) and \(\Tilde{Y}\) satisfy (\ref{strcondd}) due to
\[\Sigma^{\Tilde{Y}}_{ij}-\Sigma^X_{ij}=\Sigma^Y_{ij}+\sigma^2-\Sigma^X_{ij} \geq \Sigma^Y_{ij}+\frac{1}{2}(\Sigma^Y_{ii}+\Sigma^Y_{jj}+\Sigma^X_{ii}+\Sigma^X_{jj})-\Sigma^X_{ij}=\frac{Var(Y_i+Y_j)+Var(X_i-X_j)}{2} \geq 0,\]
\[\Sigma^{\Tilde{Y}}_{ii}-\Sigma^X_{ii}+\Sigma^{\Tilde{Y}}_{jj}-\Sigma^X_{jj}-2(\Sigma^{\Tilde{Y}}_{ij}-\Sigma^X_{ij})=\mathbb{E}[(Y_i-Y_j)^2]-\mathbb{E}[(X_i-X_j)^2] \geq 0.\]
Thus, the result for \(\Tilde{Y}\) and
\[\max_{1 \leq i \leq k}{|\Tilde{Y}_i|} \leq \sigma|g|+\max_{1 \leq i \leq k}{|Y_i|}\]
give
\[\mathbb{E}[(\max_{1 \leq i \leq k}{|X_i|})^m] \leq \mathbb{E}[(\max_{1 \leq i \leq k}{|\Tilde{Y}_i|})^m] \leq 2^{m-1} \cdot (\sigma^m\mathbb{E}[|g|^m]+\mathbb{E}[(\max_{1 \leq i \leq k}{|Y_i|})^m]),\]
entailing the claimed inequality.
\end{proof}

\begin{remark}
    While clearly
\[\mathbb{E}[(\max_{1 \leq i \leq k}{|Y_i|})^m] \geq (\max_{1 \leq i \leq k}{\sqrt{\mathbb{E}[Y_i^2]}})^{m} \cdot \mathbb{E}[|g|^m],\]
oftentimes \(\max_{1 \leq i \leq k}{\mathbb{E}[X_i^2]} \leq \max_{1 \leq i \leq k}{\mathbb{E}[Y_i^2]},\) in which case the above result gives
\[\mathbb{E}[(\max_{1 \leq i \leq k}{|X_i|})^m] \leq 2^{m-1} (2^{m/2}+1) \cdot \mathbb{E}[(\max_{1 \leq i \leq k}{|Y_i|})^m]\]
from \(\sigma \leq \sqrt{2} \cdot \max_{1 \leq i \leq k}{\sqrt{\mathbb{E}[Y_i^2]}}.\)
\end{remark}

\subsection{Proof of Theorem~\ref{extsudfer}}\label{sect01}

It suffices to show the claimed inequality when
\begin{equation}\label{str}
    Var(X_i),\hspace{0.3cm} Var(Y_i)>0,\hspace{0.3cm}|Corr(X_i,X_j)|<1,\hspace{0.3cm}|Corr(Y_i,Y_j)|<1
\end{equation}
for all \(1 \leq i,j\leq k, i \ne j.\) To see this, note that the general case 
follows from the dominated convergence theorem by employing
\[X_\epsilon=X+\epsilon \xi,\hspace{0.5cm} Y_\epsilon=Y+\epsilon \xi,\] 
for \(\xi \in \mathbb{R}^k,\xi\overset{d}{=} N(0,I)\) independent of \(X,Y,\) and letting \(\epsilon>0\) tend to \(0:\) the three conditions continue to hold as 
\[\mathbb{E}[\xi_i]=0,\hspace{0.3cm}\mathbb{E}[(X_{i}+\epsilon \xi_i)^2]=\mathbb{E}[X_{i}^2]+\epsilon^2,\hspace{0.3cm}\mathbb{E}[(X_{i}-X_j+\epsilon \xi_i-\epsilon \xi_j)^2]=\mathbb{E}[(X_{i}-X_j)^2]+2\epsilon^2,\] 
\[|Corr(X_i,X_j)|=\frac{|Cov(X_i,X_j)|}{\sqrt{(Var(X_i)+\epsilon^2) \cdot (Var(X_j)+\epsilon^2)}} \leq \sqrt{\frac{Var(X_i) \cdot Var(X_j)}{(Var(X_i)+\epsilon^2) \cdot (Var(X_j)+\epsilon^2)}}<1,\]
with the natural analogs for the entries of \(Y\) also holding. 
\par
Assume now (\ref{str}) holds. For \(p \in \{2n,n \in \mathbb{N}\},\) let \(f_p:\mathbb{R}^k \to \mathbb{R}\) be given by
\[f_p(x_1,x_2,\hspace{0.05cm}...\hspace{0.05cm},x_k)=(x_1^p+x_2^p+...+x_k^p+e^{-p^2})^{m/p}.\]
It is immediate that \(f_p\) is smooth, and for \(x_1,x_2,\hspace{0.05cm}...\hspace{0.05cm},x_k \in \mathbb{R},\)
\begin{equation}\label{bdd}
    (\max_{1 \leq i \leq k}{|x_i|})^m \leq f_p(x_1,x_2,\hspace{0.05cm}...\hspace{0.05cm},x_k) \leq 2^{m/p} \cdot [k^{m/p} \cdot (\max_{1 \leq i \leq k}{|x_i|})^m+e^{-pm}]
\end{equation}
(\(x^p \geq 0, a^\alpha \leq (a+b)^{\alpha} \leq 2^{\alpha} \cdot (a^{\alpha}+b^{\alpha})\) for \(x \in \mathbb{R},a,b,\alpha>0\)), whereby
\[\lim_{p \to \infty}{f_p(x_1,x_2,\hspace{0.05cm}...\hspace{0.05cm},x_k)}=(\max_{1 \leq i \leq k}{|x_i|})^m.\]
In light of this, it suffices to justify that
\begin{equation}\label{pin}
    \limsup_{p=2n, n \to \infty}{(\mathbb{E}[f_p(X)]-\mathbb{E}[f_p(Y)])} \leq 0
\end{equation}
since the dominated convergence theorem alongside (\ref{bdd}) entail
\[\lim_{p=2n, n \to \infty}{(\mathbb{E}[f_p(X)]-\mathbb{E}[f_p(Y)])}=\mathbb{E}[(\max_{1 \leq i \leq k}{|X_i|})^m]-\mathbb{E}[(\max_{1 \leq i \leq k}{|Y_i|})^m],\]
rendering the desired claim.
\par
It remains to prove (\ref{pin}), a key towards it being Gaussian interpolation. Fix \(p \in \{2n,n \in \mathbb{N}\},p>m/2:\) the mean-value theorem and Lemma~\ref{interpollemma} applied to \(f_p\) give
\begin{equation}\label{meant}
    \mathbb{E}[f_p(X)]-\mathbb{E}[f_p(Y)]=\frac{1}{2}\sum_{1 \leq i,j \leq d}{(\Sigma^X_{ij}-\Sigma^Y_{ij})\mathbb{E}[\frac{\partial^2f_p}{\partial{x_i} \partial{x_j}}(Z(u_p))]},
\end{equation}
for some \(u_p \in (0,1),\) where
\[Z(u)=\sqrt{u} \cdot X+\sqrt{1-u} \cdot Y.\]
It is shown next that
\[Z_{p,u}:=\sum_{1 \leq i,j \leq d}{(\Sigma^X_{ij}-\Sigma^Y_{ij})\frac{\partial^2f_p}{\partial{x_i} \partial{x_j}}(Z(u))}=Z_{p,u,-}+Z_{p,u,0}\]
with 
\begin{equation}\label{decomp}
    Z_{p,u,-} \leq 0, \hspace{0.5cm} \mathbb{E}[Z_{p,u,0}] \leq \frac{C}{p}
\end{equation}
for all \(u \in [0,1],\) and some \(C=\overline{C}(m,X,Y)>0,\) decomposition that yields (\ref{pin}) by virtue of (\ref{meant}).
\par
Begin by computing the partial derivatives of 
\[f_p(x_1,x_2,\hspace{0.05cm}...\hspace{0.05cm},x_k)=(x_1^p+x_2^p+...+x_k^p+e^{-p^2})^{m/p}:\]
for ease of notation, let
\[A=x_1^p+x_2^p+...+x_k^p+e^{-p^2}.\]
For \(1 \leq i,j \leq k,\)
\[\frac{\partial{f}_p}{\partial{x_{i}}}=A^{m/p-1} \cdot mx_i^{p-1},\]
\[\frac{\partial^2{f}_p}{\partial{x_{i}}\partial{x_j}}=A^{m/p-2} \cdot mx_i^{p-1} \cdot (m-p)x_j^{p-1}+A^{m/p-1} \cdot m(p-1)x_i^{p-2}\chi_{i=j},\]
whereby 
\[Z_{p,u}=m(p-1)A^{m/p-1}\sum_{1 \leq i \leq k}{(\Sigma^X_{ii}-\Sigma^Y_{ii})x_i^{p-2}}+m(m-p)A^{m/p-2}\sum_{1 \leq i,j \leq k}{(\Sigma^X_{ij}-\Sigma^Y_{ij})x_i^{p-1}x_j^{p-1}}.
\]
Writing \(m(m-p)=m(m-1)-m(p-1)\) leads to
\[Z_{p,u}=m(m-1)A^{m/p-2}\sum_{1 \leq i,j \leq k,i \ne j}{(\Sigma^X_{ij}-\Sigma^Y_{ij})x_i^{p-1}x_j^{p-1}}+m(m-1)A^{m/p-2}\sum_{1 \leq i \leq k}{(\Sigma^X_{ii}-\Sigma^Y_{ii})x_i^{2p-2}}+\]
\[+m(p-1)A^{m/p-2}[A \sum_{1 \leq i \leq k}{(\Sigma^X_{ii}-\Sigma^Y_{ii})x_i^{p-2}}-\sum_{1 \leq i,j \leq k}{(\Sigma^X_{ij}-\Sigma^Y_{ij})x_i^{p-1}x_j^{p-1}}]:=Z_{p,u,1}+Z_{p,u,2}+Z_{p,u,3}.\]
The claimed properties in (\ref{decomp}) ensue from the analysis below which deals separately with each of these new random variables. Concretely,
\[Z_{p,u,0}=Z_{p,u,1}+(Z_{p,u,3}-Z_{p,u,3,-}), \hspace{0.5cm} Z_{p,u,-}=Z_{p,u,2}+Z_{p,u,3,-}\]
for
\[Z_{p,u,3,-}=-m(p-1)A^{m/p-2} e^{-p^2}\sum_{1 \leq i \leq k}{\Delta_{ii}x_i^{p-2}},\]
where for ease of notation,
\[\Delta_{ij}=\Sigma^Y_{ij}-\Sigma^X_{ij} \hspace{0.5cm} (1 \leq i,j \leq k),\] 
\[M=\sum_{1 \leq i,j \leq k}{|\Delta_{i,j}|}.\]

\subsection*{Part I}

Consider \(Z_{p,u,1}:\) Lemma~\ref{lemmadelta} and (\ref{str}) entail that for \(i \ne j,\) \(0<\min{(|x_i|,|x_j|)}<\max{(|x_i|,|x_j|)}\) with probability \(1,\) whereby 
\[\mathbb{E}[Z_{p,u,1}] \leq Mm(m-1)\sum_{1 \leq i,j \leq k, i \neq j}{\mathbb{E}[(\max{(|x_i|,|x_j|)})^{m-2} \cdot (\frac{\min{(|x_i|,|x_j|)}}{\max{(|x_i|,|x_j|)}})^{p-1}]}\]
since the contribution of \((i,j),i \ne j\) is at most
\[m(m-1)A^{m/p-2}|(\Sigma^X_{ij}-\Sigma^Y_{ij})x_i^{p-1}x_j^{p-1}| \leq m(m-1) \cdot (\max{(|x_i|,|x_j|)})^{m-2p} \cdot M \cdot (\max{(|x_i|,|x_j|)})^{p-1} \cdot (\min{(|x_i|,|x_j|)})^{p-1}=\]
\[=Mm(m-1) \cdot (\max{(|x_i|,|x_j|)})^{m-2} \cdot (\frac{\min{(|x_i|,|x_j|)}}{\max{(|x_i|,|x_j|)}})^{p-1}.\]
Consequently, Lemma~\ref{lemmaint2} and Lemma~\ref{lemmadelta} give that
\[\mathbb{E}[Z_{p,u,1}] \leq Mm(m-1) \cdot k^2 \cdot \frac{C_1(m,X,Y)}{p}.\]

\subsection*{Part II}

Consider now \(Z_{p,u,2}:\) the initial condition (\ref{strcondd}) can be rewritten as 
\begin{equation}\label{deltacond}
    2|\Delta_{ij}| \leq \Delta_{ii}+\Delta_{jj}.
\end{equation}
Choosing \(i=j\) entails \(\Delta_{ii} \geq 0,\) which together with \(m \geq 1\) renders
\[Z_{p,u,2}=m(m-1)A^{m/p-2}\sum_{1 \leq i \leq k}{(\Sigma^X_{ii}-\Sigma^Y_{ii})x_i^{2p-2}}=-m(m-1)A^{m/p-2}\sum_{1 \leq i \leq k}{\Delta_{ii}x_i^{2p-2}} \leq 0.\]

\subsection*{Part III}

Consider lastly \(Z_{p,u,3}:\) its last factor is
\[-A \sum_{1 \leq i \leq k}{\Delta_{ii}x_i^{p-2}}+\sum_{1 \leq i,j \leq k}{\Delta_{ij}x_i^{p-1}x_j^{p-1}}=\]
\[=-e^{-p^2}\sum_{1 \leq i \leq k}{\Delta_{ii}x_i^{p-2}}-\sum_{1 \leq i,j \leq k}{\Delta_{ii}x_i^{p-2}x_j^p}+\sum_{1 \leq i,j \leq k}{\Delta_{ij}x_i^{p-1}x_j^{p-1}},\]
and 
\[-\sum_{1 \leq i,j \leq k}{\Delta_{ii}x_i^{p-2}x_j^p}+\sum_{1 \leq i,j \leq k}{\Delta_{ij}x_i^{p-1}x_j^{p-1}} \leq\]
\[\leq -\sum_{1 \leq i,j \leq k}{\Delta_{ii}x_i^{p-2}x_j^p}+\sum_{1 \leq i,j \leq k}{|\Delta_{ij}| \cdot |x_i|^{p-1} \cdot |x_j|^{p-1}} \leq\]
\[\leq -\sum_{1 \leq i<j \leq k}{[\Delta_{ii}x_i^{p-2}x_j^p+\Delta_{jj}x_j^{p-2}x_i^p-(\Delta_{ii}+\Delta_{ij}) \cdot |x_i|^{p-1} \cdot |x_j|^{p-1}]}\]
from \(2|p,\) \(\Delta_{ii} \geq 0,\) \(2|\Delta_{ij}| \leq \Delta_{ii}+\Delta_{jj}\) (recall (\ref{deltacond})). This gives that
\[\mathbb{E}[Z_{p,u,3}-Z_{p,u,3,-}] \leq m(p-1)\sum_{1 \leq i<j \leq k}{\mathbb{E}[A^{m/p-2} \cdot (-\Delta_{ii}x_i^{p-2}x_j^p-\Delta_{jj}x_j^{p-2}x_i^p+(\Delta_{ii}+\Delta_{ij}) \cdot |x_i|^{p-1} \cdot |x_j|^{p-1})]}.\]
Since
\[-\Delta_{ii}x_i^{p-2}x_j^p-\Delta_{jj}x_j^{p-2}x_i^p+(\Delta_{ii}+\Delta_{ij}) \cdot |x_i|^{p-1} \cdot |x_j|^{p-1}=\]
\[=-\Delta_{ii}|x_i|^{p-2} \cdot|x_j|^p-\Delta_{jj}|x_j|^{p-2} \cdot|x_i|^p+(\Delta_{ii}+\Delta_{ij}) \cdot |x_i|^{p-1} \cdot |x_j|^{p-1}=\]
\[=|x_i|^{p-2} \cdot |x_j|^{p-2} \cdot (|x_i|-|x_j|) \cdot (\Delta_{ii}|x_j|-\Delta_{jj}|x_i|),\]
it follows that
\[\mathbb{E}[Z_{p,u,3}-Z_{p,u,3,-}] \leq mM(p-1)\sum_{1 \leq i,j \leq k,i \ne j}{\mathbb{E}[(\max{(|x_i|,|x_j|)})^{m-1} \cdot (\frac{\min{(|x_i|,|x_j|)}}{\max{(|x_i|,|x_j|)}})^{p-2} \cdot (1-\frac{\min{(|x_i|,|x_j|)}}{\max{(|x_i|,|x_j|)}})]}\]
inasmuch as when \(i \ne j,\) \(0<\min{(|x_i|,|x_j|)}<\max{(|x_i|,|x_j|)}\) with probability \(1,\) whereby 
\[A^{m/p-2} \cdot |-\Delta_{ii}x_i^{p-2}x_j^p-\Delta_{jj}x_j^{p-2}x_i^p+(\Delta_{ii}+\Delta_{ij}) \cdot |x_i|^{p-1} \cdot |x_j|^{p-1}| \leq\]
\[\leq (\max{(|x_i|,|x_j|)})^{m-2p} \cdot (\max{(|x_i|,|x_j|)})^{p-2} \cdot (\min{(|x_i|,|x_j|)})^{p-2} \cdot (\max{(|x_i|,|x_j|)}-\min{(|x_i|,|x_j|)}) \cdot M\max{(|x_i|,|x_j|)}=\]
\[=M \cdot (\max{(|x_i|,|x_j|)})^{m-2} \cdot (\frac{\min{(|x_i|,|x_j|)}}{\max{(|x_i|,|x_j|)}})^{p-2} \cdot (1-\frac{\min{(|x_i|,|x_j|)}}{\max{(|x_i|,|x_j|)}}).\]
Lastly, Lemma~\ref{lemmaint2} in conjunction with Lemma~\ref{lemmadelta} provide
\[\mathbb{E}[Z_{p,u,3}-Z_{p,u,3,-}] \leq mM(p-1) \cdot k^2 \cdot \frac{C_1(m,X,Y)}{p(p-1)}.\]

\subsection{Auxiliary Lemmas}\label{sect02}

This subsection contains a simple result (Lemma~\ref{lemmadelta}) on the covariance structure of the interpolated random variable
\[Z(u)=\sqrt{u} \cdot X+\sqrt{1-u} \cdot Y \hspace{0.5cm} (u \in [0,1]),\]
as well as expectation inequalities regarding multivariate normal random variables (Lemma~\ref{lemmaint2}). These results are vital in the proof of Theorem~\ref{extsudfer}.

\begin{lemma}\label{lemmadelta}
    Let \(X,Y \in \mathbb{R}^2\) be independent centered Gaussian random vectors with nonzero entries, and take 
    \[W_1(t)=\sqrt{t}X_1+\sqrt{1-t}Y_1, \hspace{0.5cm} W_2(t)=\sqrt{t}X_2+\sqrt{1-t}Y_2 \hspace{0.5cm} (t \in [0,1]).\]
    Then for all \(t \in [0,1],\)
    \[v_{\min} \leq Var(W_1(t)),Var(W_2(t) \leq v_{\max}, \hspace{0.5cm} |Corr(W_1(t),W_2(t))| \leq \max{(|Corr(X_1,X_2)|,|Corr(Y_1,Y_2)|)},\]
    where
    \[v_{\min}=\min{(Var(X_1),Var(X_2),Var(Y_1),Var(Y_2))},\]  
    \[v_{\max}=\max{(Var(X_1),Var(X_2),Var(Y_1),Var(Y_2))}.\] 
\end{lemma}

\begin{proof}
    Take
    \[a=Var(X_1), \hspace{0.1cm} b=Var(Y_1),\hspace{0.1cm} c=Var(X_2),\hspace{0.1cm} d=Var(Y_2),\hspace{0.1cm}\rho_x=Cov(X_1,X_2),\hspace{0.1cm}\rho_y=Cov(Y_1,Y_2),\]
    and let \(t \in [0,1]\) be arbitrary. For \(1 \leq j \leq 2,\)
    \begin{equation}\label{inneqq}
        Var(W_{j}(t))=tVar(X_{j})+(1-t)Var(Y_{j}) \in [v_{\min},v_{\max}],
    \end{equation}
    while
    \[(Corr(W_1(t),W_2(t)))^2=\frac{(t\rho_{x} \sqrt{ac}+(1-t)\rho_{y}\sqrt{bd})^2}{(ta+(1-t)b) \cdot (tc+(1-t)d)} \leq\]
    \[\leq \max{(\rho^2_{x},\rho^2_{y})} \cdot \frac{(t\sqrt{ac}+(1-t)\sqrt{bd})^2}{(ta+(1-t)b) \cdot (tc+(1-t)d)} \leq \max{(\rho^2_{x},\rho^2_{y})}\]
    since
    \[1-\frac{(t\sqrt{ac}+(1-t)\sqrt{bd})^2}{(ta+(1-t)b) \cdot (tc+(1-t)d)}=\frac{t(1-t) \cdot (\sqrt{ad}-\sqrt{bc})^2}{(ta+(1-t)b) \cdot (tc+(1-t)d)} \geq 0.\]
\end{proof}

\begin{lemma}\label{lemmaint2}
    Suppose \(g \in \mathbb{R}^n\) is a Gaussian random vector with \(g\overset{d}{=}N(0,I),\) and let \(X_1,Y_1 \in \mathbb{R}\) be given by \(X_1=M_1g,Y_1=M_2g\) for deterministic \(M_1,M_2 \in \mathbb{R}^{1 \times n}-\{0\}\) such that \(|Corr(X_1,Y_1)|<1.\) Then for any \(m>0,\) 
    there exists \(C(m)>0\) such that for all \(p>1,\)
    \begin{equation}\label{bdint}
    \mathbb{E}[(\max{(|X_1|,|Y_1|)})^{m-2} \cdot \frac{(\min{(|X_1|,|Y_1|)})^{p-1}}{(\max{(|X_1|,|Y_1|)})^{p-1}}] \leq \frac{C(m,X_1,Y_1)}{p},    
    \end{equation}
    \begin{equation}\label{bdint2}
    \mathbb{E}[(\max{(|X_1|,|Y_1|)})^{m-2} \cdot \frac{(\min{(|X_1|,|Y_1|)})^{p-2}}{(\max{(|X_1|,|Y_1|)})^{p-2}} \cdot (1-\frac{\min{(|X_1|,|Y_1|)}}{\max{(|X_1|,|Y_1|)}})] \leq \frac{C(m,X_1,Y_1)}{p(p-1)},    
    \end{equation}    
    where 
    \begin{equation}\label{cdef1}
     C(m,X_1,Y_1)=C(m) \cdot \frac{(\max{(Var(X_1),Var(Y_1))})^{m/2}}{\sqrt{Var(X_1) \cdot Var(Y_1) \cdot (1-(Corr(X_1,Y_1)^2))}}. 
    \end{equation}
\end{lemma}

\begin{proof}
    Suppose without loss of generality that \(Var(X_1) \geq Var(Y_1):\) then 
     \[(X_1,Y_1)\overset{d}{=}(X,cX+Y)\]
    for some \(c \in (-1,1),\) and \(X,Y\) independent centered normal random variables. To see this, let 
    \[c=\frac{Cov(X_1,Y_1)}{Var(X_1)}, \hspace{0.5cm} X=X_1,\hspace{0.5cm} Y=Y_1-cX_1;\]
    then 
    \[|c|=|Corr(X_1,Y_1)| \cdot \sqrt{\frac{Var(Y_1)}{Var(X_1)}} \leq |Corr(X_1,Y_1)|<1,\] 
    \(X,Y\) are centered as well as independent because their joint distribution is multivariate Gaussian and
    \[Cov(X,Y)=Cov(X_1,Y_1)-cVar(X_1)=0.\]
    Since 
    \[Var(X)=Var(X_1),\]
    \[Var(Y)=Var(Y_1)-c^2Var(X)=Var(Y_1)\cdot (1-(Corr(X_1,Y_1))^2),\]
    it suffices to show
    \[\int_{\mathbb{R}^2}{(\max{(|x|,|cx+y|)})^{m-2} \cdot \frac{(\min{(|x|,|cx+y|)})^{p-1}}{(\max{(|x|,|cx+y|)})^{p-1}}\cdot e^{-c_1x^2-c_2y^2}dxdy} \leq \frac{\overline{C}_1(m)}{p} \cdot [c_1^{-m/2}+c_3^{-m/2}],\]
     \[\int_{\mathbb{R}^2}{(\max{(|x|,|cx+y|)})^{m-2} \cdot \frac{(\min{(|x|,|cx+y|)})^{p-2}}{(\max{(|x|,|cx+y|)})^{p-2}} \cdot (1-\frac{\min{(|x|,|cx+y|)}}{\max{(|x|,|cx+y|)}})\cdot e^{-c_1x^2-c_2y^2}dxdy} \leq \frac{\overline{C}_2(m)}{p(p-1)} \cdot [c_1^{-m/2}+c_3^{-m/2}],\]
    where 
    \[c_3=\begin{cases}
        \frac{c_2}{4}, \hspace{2.4cm} c=0,\\
        \frac{1}{4} \cdot \min{(\frac{c_1}{c^2},c_2)}, \hspace{0.5cm} c \ne 0,
    \end{cases}\]
    because in the current case,
    \[c=\frac{Cov(X_1,Y_1)}{Var(X_1)}, \hspace{0.5cm} c_1=\frac{1}{2Var(X)}=\frac{1}{2Var(X_1)}, \]
    \[c_2=\frac{1}{2Var(Y)}=\frac{1}{2(Var(Y_1)-c^2Var(X_1))}=\frac{1}{2Var(Y_1)(1-(Corr(X_1,Y_1))^2)},\]
    whereby
    \[c_3 \leq \frac{1}{8Var(Y_1) \cdot \max{((Corr(X_1,Y_1))^2,1-(Corr(X_1,Y_1))^2)}} \leq \frac{1}{4Var(Y_1)}.\]
    \par
    \((a)\) Begin with the first integral of interest. Symmetry yields 
    \[\int_{\mathbb{R}^2}{(\max{(|x|,|cx+y|)})^{m-2} \cdot \frac{(\min{(|x|,|cx+y|)})^{p-1}}{(\max{(|x|,|cx+y|)})^{p-1}}\cdot e^{-c_1x^2-c_2y^2}dxdy}=\]
    \[=2\int_{x \geq 0, y \in \mathbb{R}}{(\max{(x,|cx+y|)})^{m-2} \cdot \frac{(\min{(x,|cx+y|)})^{p-1}}{(\max{(x,|cx+y|)})^{p-1}}\cdot e^{-c_1x^2-c_2y^2}dxdy}.\]
    By partitioning the integration domain into two sets, \(S_{x}\) and \(S_{y},\) the claimed result becomes a consequence of (\ref{g1b}) and (\ref{g2b}) below.
    \par
    \(I.\) \(S_{x}=\{(x,y) \in \mathbb{R}^2: x \geq |cx+y|\}:\) 
    \[\int_{x \geq |cx+y|,x>0}{x^{m-2} \cdot \frac{|cx+y|^{p-1}}{x^{p-1}} \cdot e^{-c_1x^2-c_2y^2}dxdy} \leq \int_{x \geq |cx+y|,x>0}{x^{m-2} \cdot \frac{|cx+y|^{p-1}}{x^{p-1}} \cdot e^{-c_1x^2}dxdy}=\]
    \begin{equation}\label{g1b}\tag{\(pSx\)}
        =2\int_{0}^{\infty}{x^{m-1}e^{-c_1x^2} \cdot (\int_{0}^{1}{t^{p-1}dt})dx}=\frac{2}{p} \int_{0}^{\infty}{x^{m-1}e^{-c_1x^2}dx}=\frac{2c(m)}{pc_1^{m/2}}
    \end{equation}
    using the change of variables \(y=tx-cx\) with \(dy=xdt,\) and
    \[\int_{0}^{\infty}{x^{M}e^{-x^2}dx} \leq \int_{0}^{1}{x^{M}dx}+C(M)\int_{1}^{\infty}{e^{-x^2/2}dx} \leq \frac{1}{M+1}+ C(M) \cdot \frac{\sqrt{2\pi}}{2} \hspace{0.5cm} (M>-1).\]

    \par
    \(II.\) \(S_{y}=\{(x,y) \in \mathbb{R}^2: 0 \leq x<|cx+y|\}:\)
    \[\int_{|cx+y|>x \geq 0}{|cx+y|^{m-2} \cdot \frac{|x|^{p-1}}{|cx+y|^{p-1}} \cdot e^{-c_1x^2-c_2y^2}dxdy} \leq \int_{|cx+y|>x \geq 0}{|cx+y|^{m-2} \cdot \frac{|x|^{p-1}}{|cx+y|^{p-1}} \cdot e^{-c_3(cx+y)^2}dxdy}\]
    because \(c_1x^2+c_2y^2 \geq c_3(cx+y)^2,\) which is clear when \(c=0,\) and else
    \[c_1x^2+c_2y^2=\frac{c_1}{c^2}|cx|^2+c_2|y|^2 \geq \min{(\frac{c_1}{c^2},c_2)} \cdot \frac{(|cx|+|y|)^2}{4} \geq c_3(cx+y)^2.\]
    The change of variables \((x,y) \to (x,cx+y)\) gives the last integral equals 
    \begin{equation}\label{g2b}\tag{\(pSy\)}
       2\int_{z>x \geq 0}{z^{m-2} \cdot \frac{x^{p-1}}{z^{p-1}} \cdot e^{-c_3z^2}dxdz}=\frac{2}{p}\int_{0}^{\infty}{z^{m-1}e^{-c_3z^2}dz}=\frac{2c(m)}{pc_3^{m/2}}.
    \end{equation}

    \((b)\) Consider now the second integral of interest: 
    \[\int_{\mathbb{R}^2}{(\max{(|x|,|cx+y|)})^{m-2} \cdot \frac{(\min{(|x|,|cx+y|)})^{p-2}}{(\max{(|x|,|cx+y|)})^{p-2}} \cdot (1-\frac{\min{(|x|,|cx+y|)}}{\max{(|x|,|cx+y|)}})\cdot e^{-c_1x^2-c_2y^2}dxdy}.\]
    Reasoning as in part \((a)\) above yields a bound as claimed above: the primary difference is that \(\int_{0}^{1}{t^{p-1}dt}\) is replaced by 
    \[\int_{0}^{1}{t^{p-2}(1-t)dt}=\frac{1}{p-1}-\frac{1}{p}=\frac{1}{p(p-1)}.\]
\end{proof}

\bibliography{references}

\end{document}